\documentclass[aop,authoryear,preprint]{imsart}

\RequirePackage[OT1]{fontenc}
\RequirePackage{natbib, amsmath, amsthm, amssymb}
\arxiv{math.PR/0712.0688}

\startlocaldefs
\numberwithin{equation}{section}
\def\@journal{aop}
\theoremstyle{plain}
\newcommand{\eqdef}{\stackrel{d}{=}}
\newcommand{\eqas}{\stackrel{a.s.}{=}}

\newcommand{\bX}{\mathbf{X}}
\newcommand{\Zd}{\mathbb{Z}^d}
\newcommand{\SaS}{S \alpha S}
\newcommand{\mC}{\mathcal{C}}
\newcommand{\mD}{\mathcal{D}}

\newtheorem{thm}{Theorem}[section]
\newtheorem{remark}[thm]{Remark}

\newtheorem{example}{Example}[section]

\newtheorem{lemma}[thm]{Lemma}

\endlocaldefs

\begin{document}

\begin{frontmatter}

\title{ERGODIC THEORY, ABELIAN GROUPS, AND POINT PROCESSES INDUCED BY STABLE RANDOM FIELDS}
\runtitle{Point Processes Based on Stable Fields}

\begin{aug}
\author{\fnms{Parthanil} \snm{Roy}\thanksref{t1}\ead[label=e1]{roy@stt.msu.edu}}

\thankstext{t1}{Supported in part by NSF
grant DMS-0303493 and NSF training grant ``Graduate and Postdoctoral
Training in Probability and Its Applications'' at Cornell
University and by the RiskLab of the Department of Mathematics, ETH Zurich.}
\runauthor{P. Roy}

\affiliation{Michigan State University}

\address{Department of Statistics and Probability\\Michigan State University\\ East Lansing, MI 48824-1027\\USA\\
\printead{e1}\\
\phantom{E-mail: roy@stt.msu.edu}}
\end{aug}

\begin{abstract}

We consider a point process sequence induced by a stationary symmetric $\alpha$-stable $(0 < \alpha < 2)$ discrete parameter random field. It is easy to prove, following the arguments in the one-dimensional case in \cite{resnick:samorodnitsky:2004}, that if the random field is generated by a dissipative group action then the point process sequence converges weakly to a cluster Poisson process. For the conservative case, no general result is known even in the one-dimensional case. We look at a specific class of stable random fields generated by conservative actions whose effective dimensions can be computed using the structure theorem of finitely generated abelian groups. The corresponding point processes sequence is not tight and hence needs to be properly normalized in order to ensure weak convergence. This weak limit is computed using extreme value theory and some counting techniques.

\end{abstract}



\end{frontmatter}

\section{Introduction} Suppose that  $\mathbf{X}:=\{X_t\}_{t \in \mathbb{Z}^d}$ is a stationary symmetric $\alpha$-stable $(S \alpha S)$ discrete-parameter random
field. In other words, every finite linear combination $\sum_{i=1}^k c_i X_{t_i+s}$ follows an $S \alpha S$ distribution which does not depend on $s \in \Zd$. We consider the following sequence of point processes on $[-\infty,\infty] \setminus \{0\}$
\begin{equation}
N_n=\sum_{\|t\|_\infty \leq n}
\delta_{b_n^{-1}X_t}\,,\;\;\;\;n=1,2,3,\ldots
\label{defn_of_N_scaling_with_b_n}
\end{equation}
induced by the random field $\mathbf{X}$ with an aptly chosen
sequence of scaling constants $b_n \uparrow \infty$. Here $\delta_x$
denotes the point mass at $x$. We are interested in the weak
convergence of this point process sequence in the space
$\mathcal{M}$ of Radon measures on $[-\infty,\infty] \setminus \{0\}$ equipped
with the vague topology. This is important in extreme value theory
because a number of limit theorems for various functionals of
$S\alpha S$ random fields can be obtained by continuous mapping
arguments on the associated point process sequence. See, for
example, \cite{resnick:1987}, and \cite{balkema:embrechts:2007} for a background on weak convergence of point processes and its applications to extreme value theory. See also \cite{neveu:1977}, \cite{kallenberg:1983}, and \cite{resnick:2007}.

If $\{X_t\}_{t \in \mathbb{Z}^d}$ is an iid collection of random
variables with tails decaying like those of a symmetric $\alpha$
stable distribution then $\{b_n\}$ can be chosen as follows:
\begin{equation}
b_n=n^{d/\alpha}\,.     \label{usual_choice_of_b_n}
\end{equation}
With the above choice, the sequence $\{N_n\}$ converges weakly in
the space $\mathcal{M}$ to a Poisson random measure, whose intensity
blows up near zero (this is the reason why we exclude zero from the
state space) due to clustering of normalized observations. See, once again, \cite{resnick:1987}.
Cluster Poisson processes are obtained as weak limits
also for the point processes induced by a stationary stochastic
process with the marginal distributions having balanced regularly varying
tail probabilities provided the process is a
moving average (see \cite{davis:resnick:1985}) or it satisfies
some mild mixing conditions (see \cite{davis:hsing:1995}). In these works, weak limits of various functionals of the process were computed from the point process convergence by clever use of the continuous mapping theorem. See also
\cite{mori:1977} for various possible weak limits of a
two-dimensional point process induced by strong mixing sequences.

When the dependence structure is not necessarily
local or mild, finding a suitable scaling sequence and computation
of the weak limit both become challenging. As in the one-dimensional
case in \cite{resnick:samorodnitsky:2004}, we will observe that for
point processes induced by stable random fields the choice of $\{b_n\}$ depends on the
heaviness of the tails of the marginal distributions as well as on
the length of memory. In the short memory case the choice
\eqref{usual_choice_of_b_n} of normalizing constants is appropriate
whereas in the long memory case it is not. Furthermore, the observations
may cluster so much due to long memory that one may
need to normalize the sequence $\{N_n\}$ itself to ensure weak
convergence. This phenomenon was also observed in the one-dimensional
case in \cite{resnick:samorodnitsky:2004}.

This paper is organized as follows. We present some background materials on stationary symmetric $\alpha$-stable random fields in Section \ref{sec_background}. Section \ref{sec_diss_case_point_process} deals with the point processes associated with dissipative actions, i.e., point processes based on mixed moving averages. In Section \ref{sec_point_processes_and_gp_th} we state our main result on the weak convergence of the point process sequence induced by a class of random fields generated by conservative actions whose effective dimensions can be computed using group theory. This result is proved in Section \ref{sec_proof_of_main_result} using extreme value theory and counting techniques. Finally, an example is discussed in Section \ref{sec_example}. Throughout this paper, we use the notation $c_n \sim d_n$ to mean that $c_n/d_n$ converges to a positive number as $n \rightarrow \infty$.

\section{Preliminaries} \label{sec_background}

It is well-known that every $\SaS$ random field $\mathbf{X}$ has an integral representation of the form
\begin{eqnarray}
X_t&\eqdef& \int_{S} f_t(s)M(ds),\;\; t \in \mathbb{Z}^d\,,
\label{repn_integral_SaS}
\end{eqnarray}
where $M$ is an $S \alpha S$ random measure on some standard Borel
space $(S,\mathcal{S})$ with $\sigma$-finite control measure $\mu$
and $f_t \in L^\alpha(S,\mu)$ for all $t \in \mathbb{Z}^d$. See, for example, Theorem $13.1.2$ of
\cite{samorodnitsky:taqqu:1994}. The representation $(\ref{repn_integral_SaS})$ is called an integral representation of $\{X_t\}$. Without loss of generality we can also assume that the
family $\{f_t\}$ satisfies the full support assumption
\begin{eqnarray}
\mbox{Support}\Bigl( f_t,\, t \in
\mathbb{Z}^d\Bigr)=S \label{condn_full_support}
\end{eqnarray}
because we can always replace $S$ by
$S_0=\mbox{Support}\bigl( f_t,\, t \in \mathbb{Z}^d\bigr)$ in
$(\ref{repn_integral_SaS})$.

For a stationary $\{X_t\}$, using the fact that
the action of the group $\mathbb{Z}^d$ on $\{X_t\}_{t \in
\mathbb{Z}^d}$ by translation of indices preserves the law together with
certain rigidity properties of the spaces $L^\alpha,\, \alpha<2$ it has been shown in \cite{rosinski:1995} (for $d=1$) and \cite{rosinski:2000} (for a general $d$) that there always
exists an integral representation of the form
\begin{eqnarray}
f_t(s)&=&c_t(s){\left(\frac{d \mu \circ \phi_t}{d
\mu}(s)\right)}^{1/\alpha}f \circ \phi_t(s),\;\; t \in
\mathbb{Z}^d\,, \label{repn_integral_stationary}
\end{eqnarray}
where $f \in L^{\alpha}(S,\mu)$, $\{\phi_t\}_{t \in \mathbb{Z}^d}$
is a nonsingular $\mathbb{Z}^d$-action on $(S, \mu)$ (i.e., each $\phi_t:S \rightarrow S$ is measurable, $\phi_0$ is the identity map on $S$, $\phi_{t_1+t_2}=\phi_{t_1} \circ \phi_{t_2}$ for all $t_1, t_2 \in \Zd$ and each $\mu \circ \phi_t^{-1}$ is an equivalent measure of $\mu$), and $\{c_t\}_{t
\in \mathbb{Z}^d}$ is a measurable cocycle for $\{\phi_t\}$ taking
values in $\{-1, +1\}$ (i.e., each $c_t$ is a measurable map $c_t:S
\rightarrow \{-1, +1\}$ such that for all $t_1,t_2 \in \Zd$, $c_{t_1+t_2}(s)=c_{t_2}(s) c_{t_1}\big(\phi_{t_2}(s)\big)$ for $\mu$-a.a. $s \in S$).

Conversely, if $\{f_t\}$ is of the form $(\ref{repn_integral_stationary})$ then
$\{X_t\}$ defined by $(\ref{repn_integral_SaS})$ is a stationary $S\alpha S$
random field. We will say that a stationary $S\alpha S$ random field
$\{X_t\}_{t\in \mathbb{Z}^d}$ is generated by a nonsingular
$\mathbb{Z}^d$-action $\{\phi_t\}$ on $(S, \mu)$ if it has an
integral representation of the form $(\ref{repn_integral_stationary})$ satisfying
$(\ref{condn_full_support})$.

A measurable set $W \subseteq S$ is called a wandering set for the nonsingular $\Zd$-action
$\{\phi_t\}_{t \in \Zd}$ (as defined above) if $\{\phi_t(W):\;t\in \Zd\}$ is a pairwise
disjoint collection. Proposition $1.6.1$
of \cite{aaronson:1997} gives a decomposition of $S$ into two
disjoint and invariant parts as follows: $S=\mC \cup \mD$ where $\mathcal{D} = \cup_{t \in \Zd} \phi_t(W)$ for some wandering set $W \subseteq S$, and $\mathcal{C}$ has no wandering subset of positive $\mu$-measure. $\mathcal{D}$ is called the dissipative part, and
$\mathcal{C}$ is called the conservative part of the action. The action
$\{\phi_t\}$ is called conservative if $S=\mathcal{C}$ and
dissipative if $S=\mathcal{D}$. The reader is suggested to read \cite{aaronson:1997} and \cite{krengel:1985} for various ergodic theoretical notions used in this paper. Following the notations of \cite{rosinski:1995}, \cite{rosinski:2000} and \cite{roy:samorodnitsky:2008} we can obtain the following unique (in law) decomposition of the random field $\bX$
\begin{equation}
X_t \eqdef  \int_{\mC} f_t(s)M(ds)+\int_{\mD} f_t(s)M(ds)=:X^{\mC}_t+X^{\mD}_t,\;\; t \in \mathbb{Z}^d.
\label{decomp_of_X_t}
\end{equation}
into a sum of two independent random fields $\bX^\mathcal{C}$ and $\bX^\mathcal{D}$ generated by conservative and dissipative $\Zd$-actions respectively. This decomposition implies that it is enough to study stationary $\SaS$ random fields generated by conservative and dissipative actions.

Stationary stable random fields generated by conservative actions are expected to have longer
memory than those  generated by dissipative actions because
a conservative action ``does not wander too much'', and so the same values
of the random measure $M$ in \eqref{repn_integral_stationary} contribute to observations $X_t$ far
separated in $t$. The length of  memory of stable random fields determines, among other things, the rate of growth of the partial maxima sequence
\begin{equation}
M_n := \max_{\|t\|_\infty \leq n}|X_t|, \;\;\; n
=0,1,2,\ldots\,. \label{maxm}
\end{equation}
If $X_t$ is generated by a conservative action, the partial maxima sequence \eqref{maxm} grows at a slower rate because longer memory prevents
erratic changes in $X_t$ even when $t$ becomes ``large''. More specifically,
\begin{equation}
n^{-d/\alpha} M_n \Rightarrow \left\{
                                     \begin{array}{ll}
                                     c_\bX Z_\alpha & \mbox{ if $\bX$ is generated by a dissipative action} \\
                                     0              & \mbox{ if $\bX$ is generated by a conservative action}
                                     \end{array}
                              \right. \label{conv_of_M_n}
\end{equation}
weakly as $n \rightarrow \infty$. Here $Z_\alpha$ is a standard Frech\'{e}t type extreme value random variable with distribution function
\begin{equation}
P(Z_\alpha \leq x)=e^{-x^{-\alpha}},\;\,x > 0,  \label{cdf_of_Z_alpha}
\end{equation}
and $c_\bX$ is a positive constant depending on the random field $\bX$. The above dichotomy, which was established in the $d=1$ case by \cite{samorodnitsky:2004a} and in the general case by \cite{roy:samorodnitsky:2008}, implies that the choice of scaling sequence \eqref{usual_choice_of_b_n} is not appropriate in the conservative case since all the points in the sequence $\{N_n\}$ will be driven to zero by this normalization. On the other hand, we will see in the next section that \eqref{usual_choice_of_b_n} is indeed a good choice when the underlying action is dissipative.

\section{The Dissipative Case} \label{sec_diss_case_point_process}

\noindent Assume, in this section, that $\mathbf{X}$ is a stationary $S\alpha S$ discrete
parameter random field generated by a dissipative
$\mathbb{Z}^d$-action. In this case $\bX$ has the following mixed moving average representation:
\begin{eqnarray}
\mathbf{X} \eqdef \left\{\int_{W \times
{\mathbb{Z}}^d}f(v,t+s)\,M(dv,ds)\right\}_{t \in {\mathbb{Z}}^d}\,,
\label{repn_mixed_moving_avg}
\end{eqnarray}
where $f \in L^{\alpha}(W \times {\mathbb{Z}}^d, \nu \otimes \zeta)$, $\zeta$ is the counting measure on ${\mathbb{Z}}^d$, $\nu$ is a
$\sigma$-finite measure on a standard Borel space $(W,
\mathcal{W})$, and $M$ is a $\SaS$ random measure on $W \times {\mathbb{Z}}^d$ with control measure $\nu
\otimes \zeta$. Mixed moving averages were first introduced by \cite{surgailis:rosinski:mandrekar:cambanis:1993}. The above representation was established in the $d=1$ case by \cite{rosinski:1995} and in the general case by \cite{roy:samorodnitsky:2008} based on a previous work by \cite{rosinski:2000}.

Suppose $\nu_\alpha$ is the
symmetric measure on $[-\infty,\infty] \setminus \{0\}$ given by
$$\nu_\alpha(x,\infty]=\nu_\alpha[-\infty,-x)=x^{-\alpha},\;\;x>0\,.$$
Without loss of generality, we assume that the original stable
random field is of the form given in \eqref{repn_mixed_moving_avg}.
Let
\begin{equation}
N=\sum_i \delta_{(j_i,v_i,u_i)} \sim PRM(\nu_\alpha \otimes \nu
\otimes \zeta) \label{defn_of_N}
\end{equation}
be a Poisson random measure on $([-\infty,\infty] \setminus \{0\}) \times W
\times \mathbb{Z}^d$ with mean measure $\nu_\alpha \otimes \nu
\otimes \zeta$. Then from the assumption above it follows that
$\mathbf{X}$ has the following series representation:
\begin{equation}
X_t = {C_\alpha}^{1/\alpha} \sum_{i} j_i f(v_i,u_i+t),\;\;t \in
\mathbb{Z}^d\,, \label{repn_Possion_integral_X_t}
\end{equation}
where $C_\alpha$ is the stable tail constant given by
\begin{equation}
C_\alpha = {\left(\int_0^\infty x^{-\alpha} \sin{x}\,dx
\right)}^{-1}
 =\left\{
  \begin{array}{ll}
  \frac{1-\alpha}{\Gamma(2-\alpha) \cos{(\pi
\alpha/2)}},&\mbox{\textit{\small{if }}}\alpha \neq 1,\\
  \frac{2}{\pi},
&\mbox{\textit{\small{if }}}\alpha = 1.
  \end{array}
 \right. \label{C_alpha_defn}
\end{equation}
See, for example, \cite{samorodnitsky:taqqu:1994}.

It follows from \eqref{conv_of_M_n} that the partial maxima sequence
\eqref{maxm} grows exactly at the rate $n^{d/\alpha}$. As expected,
$b_n \sim n^{d/\alpha}$ turns out to be the right normalization for
the point process \eqref{defn_of_N_scaling_with_b_n} in this case. The following theorem, which is an extension of Theorem $3.1$ in \cite{resnick:samorodnitsky:2004} to the $d>1$ case, states that with this choice of $\{b_n\}$ the limiting random measure is a cluster Poisson random measure even though the
dependence structure is no longer weak or local. The proof is parallel to the one-dimensional case and hence omitted.

\begin{thm} \label{thm_point_process_Z^d_diss} Let $\mathbf{X}$ be the mixed moving average \eqref{repn_mixed_moving_avg}, and define the point process $N_n=\sum_{\|t\|_\infty \leq n} \delta_{{(2n)}^{-d/\alpha}X_t},\,n=1,2,\ldots$ . Then $N_n \Rightarrow N_\ast$ as $n \rightarrow \infty$, weakly in the space $\mathcal{M}$, where $N_\ast$ is a cluster Poisson random measure with representation
\begin{equation}
N_\ast=\sum_{i=1}^\infty \sum_{t \in \mathbb{Z}^d}\delta_{j_i f(v_i,
t)}\,, \label{defn_of_N_star}
\end{equation}
where $j_i$, $v_i$ are as in \eqref{defn_of_N}.
Furthermore, $N_\ast$ is Radon on $[-\infty,\infty] \setminus \{0\}$ with
Laplace functional ($g \geq 0$ continuous with compact support)
\begin{equation}
\psi_{N_\ast}(g)=E\left(e^{-N_\ast (g)}\right) \hspace{2.5in}
\label{form_of_laplace_functional_of_N_star}
\end{equation}
\begin{equation}
=\exp{\bigg\{-\iint_{([-\infty,\infty] \setminus \{0\})\times W}
\Big(1-e^{-\sum_{t \in \mathbb{Z}^d}g\big(x
f(v,t)\big)}\Big)\nu_\alpha(dx)\nu(dv) \bigg\}}. \nonumber
\end{equation}
\end{thm}

\begin{remark} \label{remark_on_noncons_case}
\textnormal{The above result is true as long as the underlying action is not conservative because Theorem $4.3$ in \cite{roy:samorodnitsky:2008} ensures that the conservative part of the random field (see \eqref{decomp_of_X_t}) will be killed by the normalization \eqref{usual_choice_of_b_n} and hence the mixed moving average part will determine the convergence.}
\end{remark}

\section{Point Processes and Group Theory}  \label{sec_point_processes_and_gp_th}

This section deals with the longer memory case, i.e., the random
field $\bX$ is now generated by a conservative action. In this case, we
know from \eqref{conv_of_M_n} that the partial maxima
sequence \eqref{maxm} of the random field grows at a rate slower
than $n^{d/\alpha}$. Hence \eqref{usual_choice_of_b_n} is
inappropriate in this case. In general, there may or may not exist a
normalizing sequence $\{b_n\}$ that ensures weak convergence of
$\{N_n\}$. See \cite{resnick:samorodnitsky:2004} for examples of
both kinds in the $d=1$ case.

We will work with a specific class of stable random fields generated
by conservative actions for which the effective dimension $p \leq d$
is known. For this class of random fields, the point process
$\{N_n\}$ will not converge weakly to a nontrivial limit for any
choice of the scaling sequence. Even for the most appropriate choice of
$\{b_n\}$, the associated point process won't even be tight (see
Remark \ref{remark_on_nonconv_of_N_n_in_cons_case} below) because of
the clustering effect of extreme observations due to longer memory
of the random field. Hence, in order to ensure weak convergence, we
have to normalize the point process sequence $\{N_n\}$ in addition
to using a normalizing sequence $\{b_n\}$ different from
\eqref{usual_choice_of_b_n} for the points. This phenomenon was also observed in
Example $4.2$ in the one-dimensional case in
\cite{resnick:samorodnitsky:2004}.

Without loss of generality we may assume that the original stable random field is of the form given in \eqref{repn_integral_SaS} and \eqref{repn_integral_stationary}. Following the approach of \cite{roy:samorodnitsky:2008} we view the underlying action as a group of invertible
nonsingular transformations on $(S,\mu)$ and use some basic counting
arguments to analyze the point process $\{N_n\}$.  We start with introducing the appropriate notation.

Consider $A:=\{\phi_t:\,t \in \mathbb{Z}^d\}$ as a subgroup of the
group of invertible nonsingular transformations on $(S,\mu)$ and
define a group homomorphism
\[
\Phi:\mathbb{Z}^d \rightarrow A
\]
by $\Phi(t)=\phi_t$ for all $t \in \mathbb{Z}^d$. Let
$K:=Ker(\Phi)=\{t \in \mathbb{Z}^d:\,\phi_t = 1_S\}$, where $1_S$
denote the identity map on $S$. Then $K$ is a free abelian group and
by the first isomorphism theorem of groups (see, for example,
\cite{lang:2002}) we have
\[
A \simeq \mathbb{Z}^d/K \,.
\]
Hence, by the structure theorem for finitely generated abelian groups (see Theorem $8.5$ in Chapter I of \cite{lang:2002}), we get
\[
A=\bar{F} \oplus \bar{N}\,,
\]
where $\bar{F}$ is a free abelian group and $\bar{N}$ is a finite
group. Assume $rank(\bar{F})=p \geq 1$ and $|\bar{N}|=l$. Since
$\bar{F}$ is free, there exists an injective group homomorphism
\[
\Psi: \bar{F} \rightarrow \mathbb{Z}^d
\]
such that $\Phi \circ \Psi = 1_{\bar{F}}$. Let $F=\Psi(\bar{F})$.
Then $F$ is a free subgroup of $\mathbb{Z}^d$ of rank $p$. In
particular, $p \leq d$.

The rank $p$ is the effective dimension of the random field, giving
more precise information on the choice of normalizing sequence $\{b_n\}$ than the nominal dimension $d$. Theorem $5.4$ in \cite{roy:samorodnitsky:2008} yields a better estimate on the rate of growth of the partial maxima  \eqref{maxm} than \eqref{conv_of_M_n}, namely
\begin{equation}
n^{-p/\alpha} M_n \Rightarrow \left\{
                                     \begin{array}{ll}
                                     c^\prime_\bX Z_\alpha & \mbox{ if $\{\phi_t\}_{t \in F}$ is a dissipative action} \\
                                     0              & \mbox{ if $\{\phi_t\}_{t \in F}$ is a conservative action}.
                                     \end{array}
                              \right. \label{conv_of_M_n_better_est}
\end{equation}
Here $c^\prime_\bX$ is another positive constant depending on $\bX$ and $Z_\alpha$ is as in \eqref{cdf_of_Z_alpha}. Hence we can guess that $b_n \sim
n^{p/\alpha}$ is a legitimate choice of the scaling sequence
provided $\{\phi_t\}_{t \in F}$ is dissipative.

It is easy to check that the sum $F+K$ is direct and
\begin{equation}
\mathbb{Z}^d/G \simeq  \bar{N}\,, \label{isomorph_needed_for_d=p+q}
\end{equation}
where $G=F \oplus K$. Let $x_1+G,\,
x_2+G,\,\ldots\,,x_l+G$ be all the cosets of $G$ in $\mathbb{Z}^d$. We give a group structure to
\begin{equation}
H:=\bigcup_{k=1}^l (x_k + F) \label{defn_of_H}
\end{equation}
as follows. For all $u_1, u_2 \in H$, there exists unique $u \in H$
such that $(u_1+u_2)-u \in K$. We define this $u$ to be $u_1 \oplus
u_2$. Clearly, $H$ becomes a countable abelian group isomorphic to $\mathbb{Z}^d/K$ under the operation $\oplus$ (``addition modulo $K$'').

Define a map $N:H \to
\{0,1,\ldots\}$ as,
\[
N(u):=\min\{\|u+v\|_ \infty: v \in K\}\,.
\]
It is easy to check that $N(\cdot)$ satisfies ``symmetry'': for all $u \in H$,
\begin{equation}
N(u^{-1})=N(u)\,, \label{symmetry_of_N}
\end{equation}
where $u^{-1}$ is the inverse of $u$ in $(H,\oplus)$, and the ``triangle inequality": for all $u_1, u_2
\in H$,
\begin{equation}
N(u_1 \oplus u_2) \leq N(u_1) + N(u_2)\,.
\label{triangle_inequality_of_N}
\end{equation}
Define
\begin{equation}
H_n=\{u \in H: N(u) \leq n\}\,. \label{defn_of_H_n}
\end{equation}
It has been shown in \cite{roy:samorodnitsky:2008} that the $H_n$'s are finite and
\begin{equation}
|H_n| \sim n^p\,.\label{size_of_H_n}
\end{equation}
Also, clearly $H_n \uparrow H$.

If $\{\phi_t\}_{t \in F}$ is a dissipative group action then we get
a dissipative $H$-action $\{\psi_u\}_{u \in H}$ defined by
\begin{equation}
\psi_u = \phi_u \;\;\;\mbox{ for all }u \in H\,.
\label{defn_of_psi_u}
\end{equation}
See, once again, \cite{roy:samorodnitsky:2008}. In this case, if we further assume that the
cocycle in \eqref{repn_integral_stationary} satisfies
\begin{equation}
c_t \equiv 1 \;\;\;\mbox{ for all }t \in K,
\label{assumption_on_c_t_for_t_in_K}
\end{equation}
then it will follow that $\{c_u\}_{u \in H}$ is an $H$-cocycle for
$\{\psi_u\}_{u \in H}$, i.e., for all $u_1,u_2 \in H$,
\[
c_{u_1 \oplus
u_2}(s)=c_{u_1}(s)c_{u_2}\big(\psi_{u_1}(s)\big)\;\;\mbox{for
$\mu$-a.a. }s \in S.
\]
Hence the subfield $\{X_u\}_{u \in H}$ is $H$-stationary and is
generated by the dissipative action $\{\psi_u\}_{u \in H}$. This implies, in particular, that there is a standard Borel space
$(W,\mathcal{W})$ with a $\sigma$-finite measure $\nu$ on it such
that
\begin{equation}
X_u \eqdef \int_{W \times H} h(w,u \oplus s)\,M^\prime(dw,ds),\;\;\;
u \in H, \label{mixed_moving_avg_repn_of_X_u_u_in_H}
\end{equation}
for some $h \in L^{\alpha}(W \times H, \nu \otimes \tau)$, where
$\tau$ is the counting measure on $H$, and  $M^\prime$ is a $S\alpha
S$ random measure on $W \times H$ with control measure $\nu \otimes
\tau$ (see, for example, Remark $2.4.2$ in \cite{roy:2008}).

Once again, we may assume, without loss of generality, that the
original subfield $\{X_u\}_{u \in H}$ is given in the form
\eqref{mixed_moving_avg_repn_of_X_u_u_in_H}. Let
\begin{equation}
N^\prime=\sum_i \delta_{(j_i,v_i,u_i)} \sim PRM(\nu_\alpha \otimes
\nu \otimes \tau) \label{defn_of_N^prime}
\end{equation}
be a Poisson random measure on $([-\infty,\infty] \setminus \{0\}) \times W
\times H$ with mean measure $\nu_\alpha \otimes \nu \otimes \tau$.
The following series representation holds in parallel to
\eqref{repn_Possion_integral_X_t}:
\begin{equation}
X_u = C_\alpha^{1/\alpha}\sum_{i=1}^\infty j_ih(v_i,u_i\oplus
u),\;\;\;u \in H,  \label{series_repn_of_X_u_u_in_H}
\end{equation}
where $C_\alpha$ is the stable tail constant \eqref{C_alpha_defn}.

Let $rank(K)=q \geq
1$ (we can also allow $q=0$ provided we follow the convention
mentioned in Remark $\ref{remark_on_q=0})$. Note that from \eqref{isomorph_needed_for_d=p+q} it follows that $q=d-p$. Choose a basis $\{\bar{u}_1,\bar{u}_2,\ldots,\bar{u}_p\}$ of $F$ and a basis $\{\bar{v}_1,\bar{v}_2,\ldots,\bar{v}_q\}$
of $K$. Let $U$ be the $d \times p$ matrix with $\bar{u}_i$ as its
$i^{th}$ column and V be the $d \times q$ matrix with $\bar{v}_j$ as its
$j^{th}$ column. Define
\begin{equation}
C=\{y \in \mathbb{R}^p: \mbox{ there exists }\lambda \in
\mathbb{R}^q \mbox{ such that } \|Uy+V\lambda\|_\infty \leq 1 \}\,. \nonumber
\end{equation}
Let $|C|$ denote the $p$-dimensional volume of $C$, and for $y \in
C$ denote by $\mathcal{V}(y)$ the $q$-dimensional volume of the
polytope
\[
P_y:=\{\lambda \in \mathbb{R}^q: \|Uy+V\lambda\|_\infty \leq 1\}\,.
\]
Define, for $t \in H$,
\begin{equation}
m(t,n):=\big|[-n\mathbf{1}, n\mathbf{1}] \cap (t+K)\big|\,.
\label{defn_of_m(t)}
\end{equation}
Here $|B|$ denotes the cardinality of the finite set $B$, $\mathbf{1}=(1,1,\ldots,1) \in \mathbb{Z}^d$, and for $u = (u^{(1)}, u^{(2)}, \ldots, u^{(d)})$ and $v =(v^{(1)}, v^{(2)}, \ldots, v^{(d)})$,
\[
[u,v]:=\left \{(t^{(1)}, t^{(2)}, \ldots, t^{(d)}) \in \mathbb{Z}^d: u^{(i)} \leq t^{(i)} \leq v^{(i)} \mbox{ for all }1 \leq i \leq d\right \}.
\]
The following result, which is an extension of Theorem
\ref{thm_point_process_Z^d_diss} (see Remark \ref{remark_on_q=0}
below), states that the weak limit of properly scaled $\{N_n\}$ is a
random measure which is not a point process.
\begin{thm} \label{thm_point_process_Gdiss}
Suppose $\{\phi_t\}_{t \in F}$ is a dissipative group action and
\eqref{assumption_on_c_t_for_t_in_K} holds. Let $\tilde{N}_n =
n^{-q} \sum_{\|t\|_\infty \leq n}
\delta_{(cn)^{-p/\alpha}X_t},\,n=1,2,\ldots$ where
$c=\big(l|C|\big)^{1/p}$. Then $\tilde{N}_n \Rightarrow
\tilde{N}_\ast$ weakly in $\mathcal{M}$, where $\tilde{N}_\ast$ is a
random measure with the following representation
\begin{equation}
\tilde{N}_\ast=\sum_{i=1}^\infty \sum_{u \in H}
\mathcal{V}(\xi_i)\delta_{j_i h(v_i, u)}\,,
\end{equation}
where $\{j_i\}$ and $\{v_i\}$ are as in \eqref{defn_of_N^prime},
$\{\xi_i\}$ is a sequence of iid $p$-dimensional random vectors
uniformly distributed in $C$ independent of $\{j_i\}$ and $\{v_i\}$,
and $\mathcal{V}$ is the continuous function defined on $C$ as
above. Furthermore, $\tilde{N}_\ast$ is Radon on
$[-\infty,\infty] \setminus \{0\}$ with Laplace functional ($g \geq 0$
continuous with compact support)
\begin{equation}
\psi_{\tilde{N}_\ast}(g)=E\left(e^{-\tilde{N}_\ast (g)}\right)
\hspace{2.5in} \label{form_of_laplace_functional_of_tilde_N_star}
\end{equation}
\begin{equation*}
=\exp{\left\{-\frac{1}{|C|}\int_{C} \int_{|x|>0} \int_W
(1-e^{-\mathcal{V}(y)\sum_{w \in H} g(xh(v,w))}) \nu(dv)
\nu_\alpha(dx) dy\right\}} \;.
\end{equation*}

\end{thm}

\begin{remark} \label{remark_on_q=0}
\textnormal{In the above theorem we can also allow $q$ to be equal
to $0$ provided we follow the convention $\mathbb{R}^0=\{0\}$, which
is assumed to have $0$-dimensional volume equal to $1$. With these
conventions, Theorem \ref{thm_point_process_Gdiss} reduces to
Theorem \ref{thm_point_process_Z^d_diss} when $q=0$. Also, by a reasoning similar to Remark \ref{remark_on_noncons_case} and using Theorem $5.4$ in \cite{roy:samorodnitsky:2008}, one can extend this result to the case when $\{\phi_t\}_{t \in F}$ is not conservative.}
\end{remark}

\begin{remark}[Due to Jan Rosi\'nski]\label{remark_jan} \textnormal{Suppose that $\{\phi_t\}_{t\in\mathbb{Z}^d}$ in \eqref{repn_integral_stationary} is measure-preserving. Define $\tilde{S}:=\{-1,1\} \times S$ and $\tilde{\mu}:=\frac{\delta_{-1}+\delta_1}{2} \otimes \mu$. Then $\psi_t(\varepsilon, s):=\left(\varepsilon c_t(s), \phi_t(s)\right),\,t \in \mathbb{Z}^d$
is a measure-preserving action on $\tilde{S}$ and
$$
X_t \eqdef \int_{\tilde{S}}\tilde{f}(\psi_t(\varepsilon,s))\tilde{M}(d\varepsilon, ds),\;\;t \in \mathbb{Z}^d,
$$ 
where $\tilde{f}(\varepsilon,s):=\varepsilon f(s) \in L^{\alpha}(\tilde{S},\tilde{\mu})$ and $\tilde{M}$ is an $S\alpha S$ random measure on $\tilde{S}$ with control measure $\tilde{\mu}$. This means, in particular, that \eqref{assumption_on_c_t_for_t_in_K} holds. Since all the known stationary $S\alpha S$ random fields are generated by actions that preserve the underlying measure (or an equivalent measure), it follows that \eqref{assumption_on_c_t_for_t_in_K} is not at all a big restriction.
}

\end{remark}

\begin{remark} \label{remark_on_nonconv_of_N_n_in_cons_case} \textnormal{Note that the above theorem together with Lemma
$3.20$ in \cite{resnick:1987} implies that the sequence of point
process \eqref{defn_of_N_scaling_with_b_n} with the choice $b_n \sim
n^{p/\alpha}$ is not tight and hence does not converge weakly in
$\mathcal{M}$. Furthermore, $\{N_n\}$ will not converge weakly to a
nontrivial limit for any other choice of normalizing sequence
$\{b_n\}$. All the points of $\{N_n\}$ will be driven to zero if
$b_n$ grows faster than $n^{p/\alpha}$. This follows from
\eqref{conv_of_M_n_better_est}, which also implies that if we select
$b_n$ to grow slower than $n^{p/\alpha}$ then we will see an
accumulation of mass at infinity. Only $b_n \sim n^{p/\alpha}$
places the points at the right scale, but they repeat so much due to
long memory, that the point process itself has to be normalized by
$n^q$ (the order of the cluster sizes) to ensure weak convergence.}
\end{remark}

\section{Proof of Theorem \ref{thm_point_process_Gdiss}} \label{sec_proof_of_main_result}

The major steps of the proof of Theorem \ref{thm_point_process_Gdiss} are similar to those of the proof of Theorem $3.1$ in \cite{resnick:samorodnitsky:2004}. However, Theorem \ref{thm_point_process_Gdiss} needs some counting which is taken care of mostly by the following lemma about $C$, $\mathcal{V}(y)$ and $m(t,n)$ defined in Section \ref{sec_point_processes_and_gp_th}.

\begin{lemma} \label{lemma_on_C_and_V} With the notations introduced above, we have:\\
(i) $C$ is compact and convex.\\
(ii) $\mathcal{V}(y)$ is a continuous function of $y$.\\
(iii) For all $1 \leq k \leq l$, the functions $m_{k,n}:C \rightarrow \mathbb{R}$ defined by
\[
m_{k,n}(y):=\frac{m\left(x_k+\sum_{i=1}^p
[ny_i]u_i,n\right)}{n^q},\;\;\;n=1,2,\ldots
\]
\big($y=(y_1,\ldots,y_p)$ \big) are uniformly bounded on $C$ and converge (as $n \rightarrow \infty$) to $\mathcal{V}(y)$ for all $y \in C$. \vspace{0.01in}\\
(iv) There is a constant $\kappa_0 >0$ such that $m(t,n)/n^q \leq
\kappa_0$ for all $t\in H$ and for all $n \geq 1$. Also,
\[
\frac{1}{n^p} \sum_{u \in H_n} \frac{m(u,n)}{n^q} \rightarrow l
\int_C \mathcal{V}(y)dy < \infty\,
\]
as $n \rightarrow \infty$. Here $H_n$ is as in \eqref{defn_of_H_n}.
\end{lemma}

\begin{proof} (i) Let $W=[U:V]$ and $z= \left[\begin{array}{cc}
                               y \\
                               \lambda
                              \end{array}
                      \right]$\,. Then $C$ is a projection of the closed and convex set
\[
P:=\{z \in \mathbb{R}^{p+q}: \|Wz\|_\infty \leq 1 \} \,.
\]
To complete the proof of part (i) it is enough to establish that $P$
is bounded. To this end note that the columns of $W$ are
independent over $\mathbb{Z}$ and hence over $\mathbb{Q}$ which
means that there is a $(p+q) \times d$ matrix $Z$ over $\mathbb{Q}$ such that $ZW=I_{p+q}$,
the identity matrix of order $p+q$. From the string of inequalities (for $z \in P$)
\begin{equation*}
\|z\|_\infty = \|ZWz\|_\infty \leq \|Z\|_\infty \|Wz\|_\infty \leq
\|Z\|_\infty
\end{equation*}
the boundedness of $P$ follows. \\

\noindent (ii) Take $\{y^{(n)}\} \subseteq C$ such that $y^{(n)}
\rightarrow y$. Fixing an integer $m \geq 1$  we get that for large
enough $n$, $\|y^{(n)}-y\| \leq \frac{1}{m}$
and hence
\begin{eqnarray*}
&&\left\{\lambda \in \mathbb{R}^q: \|Uy+V\lambda\|_\infty \leq 1-\frac{\|U\|_\infty}{m}\right\} \\
&& \subseteq P_{y^{(n)}} \subseteq \left\{\lambda \in \mathbb{R}^q:
\|Uy+V\lambda\|_\infty \leq 1+\frac{\|U\|_\infty}{m}\right\}\,.
\end{eqnarray*}
First taking the lim sup (and lim inf) as $n \rightarrow \infty$ and
then taking the limit as $m \rightarrow \infty$ we get that
\[
\mathcal{V}(y) \leq \liminf_{n \rightarrow \infty}
\mathcal{V}\big(y^{(n)}\big) \leq \limsup_{n \rightarrow \infty}
\mathcal{V}\big(y^{(n)}\big) \leq \mathcal{V}(y)
\]
which proves part (ii).\\

\noindent (iii) Fix $1 \leq k \leq l$. Let $L=\max_{1 \leq k \leq
l}\|x_k\|_\infty$. We start by showing that for all $y \in C$
\begin{equation}
m_{k,n}(y) \rightarrow \mathcal{V}(y)  \label{limit_of_m(t_in_F)}
\end{equation}
as $n \rightarrow \infty$. Let
\begin {align}
&B_n := \bigg\{\nu \in \mathbb{Z}^q: \big\|x_k+\sum_{i=1}^p
[ny_i]\bar{u}_i+V\nu \big\|_\infty \leq n \bigg\}\,,\;\;\;n \geq 1\,.
\nonumber \intertext{Since the columns of $V$ are linearly
independent over $\mathbb{Z}$, we have} &|B_n|=\big|[-n\mathbf{1},
n\mathbf{1}] \cap (x_k+\sum_{i=1}^p [ny_i]\bar{u}_i+K)\big|=n^q \,
m_{k,n}(y)\,. \label{equality_of_size_of_B_n} \intertext{Define}
&C_m:= \Big\{\lambda \in \mathbb{R}^q: \|Uy+V\lambda\|_\infty \leq 1
-\frac{1}{m}(\sum_{i=1}^p \|\bar{u}_i\|_\infty +L)\Big\}\,,\;\;\;m \geq
1\,. \nonumber
\end{align}
We first fix $m \geq 1$ and claim that for all $n \geq m$
\begin{equation}
\mathbb{Z}^q \cap nC_m \subseteq \mathbb{Z}^q \cap nC_n \subseteq
B_n \,. \label{inequality_of_B_n}
\end{equation}
The first inclusion is obvious. To prove the second one take
\[
\tilde{\nu} \in \mathbb{Z}^q \cap nC_n = \bigg\{\nu \in \mathbb{Z}^q:
\big\|\sum_{i=1}^p ny_i\,\bar{u}_i+V\nu \big\|_\infty \leq n-\sum_{i=1}^p
\|\bar{u}_i\|_\infty-L\bigg\}
\]
and observe that
\begin{eqnarray*}
&&\big\|x_k+\sum_{i=1}^p
[ny_i]\bar{u}_i+V\tilde{\nu} \big\|_\infty \\
&\leq & \|x_k\|_\infty + \big\|\sum_{i=1}^p ny_i\,\bar{u}_i+V\tilde{\nu}
\big\|_\infty + \sum_{i=1}^p \|\bar{u}_i\|_\infty \leq n\,.
\end{eqnarray*}
It follows from \eqref{inequality_of_B_n} and \eqref{equality_of_size_of_B_n} that
\begin{equation}
\frac{|\mathbb{Z}^q \cap nC_m|}{n^q} \leq
\frac{|B_n|}{n^q}=m_{k,n}(y) \, \label{inequality_of_size_of_B_n}
\end{equation}
for all $n \geq m$. Since $C_m$ is a rational polytope (i.e., a
polytope whose vertices have rational coordinates) the left hand
side of \eqref{inequality_of_size_of_B_n} converges to
$Volume(C_m)$, the $q$-dimensional volume of $C_m$ by Theorem $1$ of
\cite{deloera:2005}. Hence \eqref{inequality_of_size_of_B_n} yields
\[
Volume(C_m) \leq \liminf_{n \rightarrow \infty} m_{k,n}(y)\,.
\]
Now taking another limit as $m \rightarrow \infty$ we get
\begin{equation}
\mathcal{V}(y) \leq \liminf_{n \rightarrow \infty} m_{k,n}(y)
\label{liminf_V(y)}
\end{equation}
since $C_m \uparrow P_y$. Defining another sequence of rational
polytopes
\[
C^{\prime}_m:= \Big\{\lambda \in \mathbb{R}^q:
\|Uy+V\lambda\|_\infty \leq 1 +\frac{1}{m}(\sum_{i=1}^p
\|\bar{u}_i\|_\infty +L)\Big\}\,,\;\;\;m \geq 1
\]
and observing that $C^{\prime}_m \downarrow P_y$ as $m\rightarrow
\infty$ we can conclude using a similar argument that
\begin{equation}
\limsup_{n \rightarrow \infty} m_{k,n}(y) \leq \mathcal{V}(y) \,.
\label{limsup_V(y)}
\end{equation}
\eqref{limit_of_m(t_in_F)} follows from \eqref{liminf_V(y)} and
\eqref{limsup_V(y)}.

To establish the uniform boundedness let $R:=\sup_{y \in C}
\|y\|_\infty < \infty$ by part (i). Once again fixing $y \in C$ observe that for $C^{\prime}_1$ defined above we have
\[
C^{\prime}_1 \subseteq \Big\{\lambda \in \mathbb{R}^q:
\|V\lambda\|_\infty \leq 1 +\sum_{i=1}^p \|\bar{u}_i\|_\infty
+L+R\|U\|_\infty\Big\} =:C^\prime
\]
which is another rational polytope. Hence
\[
m_{k,n}(y) \leq \frac{|\mathbb{Z}^q \cap nC^\prime_1|}{n^q} \leq
\frac{|\mathbb{Z}^q \cap nC^\prime|}{n^q}
\]
from which the uniform boundedness follows by another application of Theorem $1$ of \cite{deloera:2005}. \\

\noindent (iv) To establish this part, we start by proving two set
inclusions which will be useful once more later in this section. For
$1 \leq k \leq l$ and $n \geq 1$ define
\begin{align*}
& F_{k,n} =\bigl\{u \in x_k+F: \mbox{ there exists } v \in K \mbox{
such that } u+v \in [-n\mathbf{1}, n\mathbf{1}]\bigr\}\\
\intertext{and}
& Q^{(k)}_n =\{\alpha \in
\mathbb{Z}^p:\,x_k+U\alpha \in F_{k,n}\}\,.
\end{align*}
Clearly
\begin{equation}
H_n=\cup_{k=1}^l F_{k,n}. \label{form_of_H_n_in_terms_of_F_k,n}
\end{equation}
Let $L=\max_{1\leq k \leq l} \|x_k\|_\infty$ as before and
$L^\prime=L+\sum_{i=1}^{p}\|u_i\|_\infty +
\sum_{j=1}^{q}\|v_j\|_\infty$. We claim that for all $n>L^\prime$,
\begin{align}
&\left\{\big([(n-L^\prime)y_1],\ldots,[(n-L^\prime)y_p]\big):\,y\in C\right\} \nonumber \\
& \;\;\subseteq Q^{(k)}_n \subseteq
\left\{\big([(n+L)y_1],\ldots,[(n+L)y_p]\big):\,y\in C\right\}\,.
\label{inclusions_of_Q_n}
\end{align}
To prove the first inclusion, let $y \in C$. Find $\lambda \in
\mathbb{R}^q$ be such that
\[
\|Uy+V\lambda\|_\infty \leq 1\,.
\]
Then we have
\begin{align*}
&\bigg\|x_k+\sum_{i=1}^{p}[(n-L^\prime)y_i]\bar{u}_i+\sum_{j=1}^{q}[(n-L^\prime)\lambda_j]\bar{v}_j\bigg\|_\infty\\
&\;\;\leq L+(n-L^\prime)\bigg\|\sum_{i=1}^{p}y_i \bar{u}_i+\sum_{j=1}^{q}\lambda_j \bar{v}_j\bigg\|_\infty+\sum_{i=1}^{p}\|\bar{u}_i\|_\infty + \sum_{j=1}^{q}\|\bar{v}_j\|_\infty\\
&\;\;\leq n
\end{align*}
proving $x_k+\sum_{i=1}^{p}[(n-L^\prime)y_i]\bar{u}_i \in F_{k,n}$ and
hence the first inclusion in \eqref{inclusions_of_Q_n}. The second
one is easy. If $\alpha \in Q^{(k)}_n$ then for some $\beta \in
\mathbb{Z}^q$
\[
\|x_k+U\alpha+V\beta\|_\infty \leq n\,,
\]
and hence
\[
\|U\alpha+V\beta\|_\infty \leq  n+L\,,
\]
which yields $y=(1/(n+L))\alpha \in C$ and establishes the second
set inclusion in \eqref{inclusions_of_Q_n}.

To prove the uniform boundedness in part (iv) we use
\eqref{inclusions_of_Q_n} as follows:
\begin{align*}
&\sup_{n \geq 1}\sup_{t\in H}\frac{m(t,n)}{n^q}\\
=&\sup_{n \geq 1}\max_{t\in H_n}\frac{m\big(t,n\big)}{n^q}\\
\leq & \max_{1 \leq k \leq l} \sup_{n \geq 1}\max_{\alpha \in Q^{(k)}_n} \frac{m(x_k+U\alpha,n+L)}{n^q}\\
\leq & \max_{1 \leq k \leq l}\sup_{n \geq 1} \sup_{y\in
C}\left(1+\frac{L}{n}\right)^q
\frac{m\big(x_k+\sum_{i=1}^{p}[(n+L)y_i]\bar{u}_i,n+L\big)}{(n+L)^q}\,,
\end{align*}
and this is bounded above by
\begin{equation}
\kappa_0=(1+L)^q \max_{1 \leq k \leq l}\sup_{n \geq 1} \sup_{y\in C}
m_{k,n}(y) \label{defn_of_kappa_0}
\end{equation}
which is finite by part (iii).

Now we prove the convergence in part (iv). Because of \eqref{form_of_H_n_in_terms_of_F_k,n} it is enough to show that for all $1 \leq k \leq l$
\begin{equation}
\frac{1}{n^p} \sum_{u \in F_{k,n}} \frac{m(u,n)}{n^q} \rightarrow
\int_C \mathcal{V}(y)dy \;\;\;(n\rightarrow \infty)\,.
\label{limit_of_avg_of_m_by_n^q_when_u_in_F_kn}
\end{equation}
To prove \eqref{limit_of_avg_of_m_by_n^q_when_u_in_F_kn} we use
\eqref{inclusions_of_Q_n} once again to get the following bound:
\begin{align*}
&\frac{1}{n^p} \sum_{u \in F_{k,n}} \frac{m(u,n)}{n^q} \\
\leq &\left(\frac{n+L}{n}\right)^p \frac{1}{(n+L)^p} \sum_{\alpha \in Q^{(k)}_n} \frac{m(x_k+U\alpha,n+L)}{n^q} \\
\leq &\left(\frac{n+L}{n}\right)^{p+q} \int_C
\frac{m\left(x_k+\sum_{i=1}^{p}[(n+L)y_i]\bar{u}_i,n+L\right)}{(n+L)^q}\,
dy + o(1)\,,
\end{align*}
from which using part (iii) and the dominated convergence theorem we get
\[
\limsup_{n\rightarrow \infty}\frac{1}{n^p} \sum_{u \in F_{k,n}}
\frac{m(u,n)}{n^q} \leq \int_C \mathcal{V}(y)dy\,.
\]
Similarly we can also prove
\[
\liminf_{n\rightarrow \infty}\frac{1}{n^p} \sum_{u \in F_{k,n}}
\frac{m(u,n)}{n^q} \geq \int_C \mathcal{V}(y)dy\,.
\]
\eqref{limit_of_avg_of_m_by_n^q_when_u_in_F_kn} follows from the
above two inequalities. This completes the proof of Lemma
\ref{lemma_on_C_and_V}.
\end{proof}

With the above lemma we are now well-prepared to prove Theorem \ref{thm_point_process_Gdiss}. Following \cite{resnick:samorodnitsky:2004}, we start with the Laplace functional of $\tilde{N}_\ast$,
\begin{eqnarray*}
\psi_{\tilde{N}_\ast}(g)&=& E\left(e^{-\tilde{N}_\ast (g)}\right)\\
                        &=& E\exp\left\{-\sum_{i=1}^\infty \sum_{u \in H}\mathcal{V}(\xi_i)g(j_i h(v_i,u))\right\}
\end{eqnarray*}
which can be shown to be equal to \eqref{form_of_laplace_functional_of_tilde_N_star} using
$$\sum_i \delta_{(j_i,v_i,\xi_i)} \sim PRM\left(\nu_\alpha \otimes \nu \otimes
\frac{1}{|C|}Leb|_C\right)$$
and by the argument used in the computation of the Laplace functional of the limiting point process in Theorem $3.1$ of \cite{resnick:samorodnitsky:2004}.

To prove that $\tilde{N}_\ast$ is Radon we take
$\eta(x)=I_{[-\infty,-\delta]\cup[\delta, \infty]},\;\delta >0$ and
look at
\begin{align}
E\left(\tilde{N}_\ast(\eta)\right)&=E \sum_{i=1}^\infty \sum_{u \in
H} \mathcal{V}(\xi_i)\eta(j_i h(v_i, u)) \nonumber \\
                                  &\leq \|\mathcal{V}\|_\infty \, E\sum_{i=1}^\infty \sum_{u \in H} \eta(j_i h(v_i, u))\,, \label{inequality_Radon}
\end{align}
where $\|\mathcal{V}\|_\infty:=\sup_{y \in C} \mathcal{V}(y) <
\infty$ by Lemma \ref{lemma_on_C_and_V}. It is enough to show that
$E\left(\tilde{N}_\ast(\eta)\right) < \infty$ which follows from \eqref{inequality_Radon} by the exact same
argument used to establish that the limiting point process in Theorem $3.1$ of \cite{resnick:samorodnitsky:2004} is Radon.

Observe that because of \eqref{assumption_on_c_t_for_t_in_K} and the assumption that the original stable random field is of the form given in \eqref{repn_integral_SaS} and \eqref{repn_integral_stationary} it follows that for all $u \in H$ and for all $v \in K$
\[
X_{u+v} \eqas X_u.
\]
As a consequence, $\tilde{N}_n$ can also be written as
\begin{equation*}
\tilde{N}_n=\sum_{t \in
H_n}\frac{m(t,n)}{n^q}\,\delta_{(cn)^{-p/\alpha}X_t}
\end{equation*}
where $m(t,n)$ is as in \eqref{defn_of_m(t)} and $H_n$ is as in
\eqref{defn_of_H_n}. The weak convergence of $\tilde{N}_n$ is
established in two steps in parallel to the proof of Theorem $3.1$ in \cite{resnick:samorodnitsky:2004} as follows: we first show that
\begin{equation*}
\tilde{N}^{(2)}_n :=\sum_{i=1}^\infty \sum_{t \in
H_n}\frac{m(t,n)}{n^q}\,\delta_{(cn)^{-p/\alpha}j_ih(v_i,u_i \oplus
t)}
\end{equation*}
converges to $\tilde{N}_\ast$ weakly in $\mathcal{M}$ and then show
that $\tilde{N}_n$ must have the same weak limit as $\tilde{N}^{(2)}_n$.

We start by proving the weak convergence of $\tilde{N}^{(2)}_n$. The scaling property of $\nu_\alpha$ yields the
Laplace functional of $\tilde{N}^{(2)}_n$ ($g \geq 0$ continuous
with compact support) as
\begin{eqnarray}
&&\;\;\;\;\;\;\;\;\;\;\;\;E\left(e^{-\tilde{N}^{(2)}_n(g)}\right) \hspace{3in} \label{form_of_laplace_functional_of_tilde_N^2_star}\\
&&=\exp{\Bigg\{-\frac{1}{(cn)^p}\int_{|x|>0}\int_W\sum_{u \in H}\Big(1-e^{-\frac{1}{n^q}\sum_{t \in H_n}m(t,n)\,g(xh(v,u\oplus t))}\Big)} \nonumber \\
&&\hspace{3.67in}\nu(dv)\nu_\alpha(dx)\Bigg\}  \nonumber
\end{eqnarray}
which needs to be shown to converge to
\eqref{form_of_laplace_functional_of_tilde_N_star}. As in \cite{resnick:samorodnitsky:2004} we
first assume that $h$ is compactly supported i.e., for some positive
integer $M$
\begin{equation}
h(v,u)I_{W \times H_M^c}(v,u) \equiv 0\,.
\label{assumption_cpt_support_h}
\end{equation}
Recall that each $H_M$ is finite
and $H_M \uparrow H$ as $M \rightarrow \infty$. Using properties
\eqref{symmetry_of_N}, \eqref{triangle_inequality_of_N} and the
compact support assumption \eqref{assumption_cpt_support_h} the
integral in \eqref{form_of_laplace_functional_of_tilde_N^2_star}
becomes
\begin{align*}
&\frac{1}{(cn)^p}\iint\sum_{u \in H_{n+M}} \left(1-\exp\left(-\sum_{t \in H_n}\frac{m(t,n)}{n^q}g\big(x h(v,u\oplus t)\big)\right)\right)\\
&\hspace{3.75in}\nu(dv)\nu_\alpha(dx) \\
\intertext{which, by a change of variable, equals}
&\frac{1}{(cn)^p}\iint\sum_{u \in H_{n+M}} \left(1-\exp\left(-\sum_{w \in A^\prime_n}\frac{m(w \ominus u,n)}{n^q}g\big(x h(v,w)\big)\right)\right)\\
&\hspace{3.85in}\nu(dv)\nu_\alpha(dx)\\
&=:\mathcal{I}_n\,.
\end{align*}
Here $w \ominus u:=w \oplus u^{-1}$, $u^{-1}$ is the inverse of $u$ in $(H,\oplus)$, and $A^\prime_n=H_M \cap
\{w^\prime:\,w^\prime \ominus u \in H_n\}$.

We claim that for all $n > M$
\begin{equation}
m(u^{-1},n-M) \leq m(w\ominus u, n) \leq m(u^{-1}, n+M)
\label{inequalities_of_m}\,.
\end{equation}
The first inequality follows, for example, because
\[
\tau \in [-(n-M)\mathbf{1},(n-M)\mathbf{1}] \cap (u^{-1}+K)
\]
if and only if
\[
\tau + w \in [-n\mathbf{1},n\mathbf{1}] \cap \big((w\ominus
u)+K\big)\,.
\]
Similarly we can prove the second inequality in
\eqref{inequalities_of_m}.

We bound $\mathcal{I}_n$ using \eqref{inequalities_of_m} by
\begin{align}
&\frac{1}{(cn)^p}\iint\sum_{u \in H_{n+M}} \left(1-\exp\left(-\sum_{w \in A^\prime_n}\frac{m(u^{-1},n+M)}{n^q}g\big(x h(v,w)\big)\right)\right) \nonumber\\
&\hspace{4in}\nu(dv)\nu_\alpha(dx) \nonumber\\
\intertext{which we claim to be equal to}
&=\frac{1}{(cn)^p}\iint\sum_{u \in H_{n+M}} \left(1-\exp\left(-\sum_{w \in A^\prime_n}\frac{m(u^{-1},n)}{n^q}g\big(x h(v,w)\big)\right)\right) \nonumber\\
&\hspace{3.9in}\nu(dv)\nu_\alpha(dx) \nonumber\\
&\;\;\;\;\;\;\;\;\;\;\;\;\;\;\;+o(1)=:\mathcal{I}^\prime_n+o(1)\,.\label{difference_of_I_n's}
\end{align}
To prove this claim observe that using the
inequality $|e^{-a}-e^{-b}| \leq |a-b|$, $(a,b>0)$ the difference of
the two integrals above can be bounded by
\begin{align*}
&\frac{1}{(cn)^p}\sum_{u \in H_{n+M}}\left(\frac{m(u^{-1},n+M)-m(u^{-1},n)}{n^q}\right)\\
&\hspace{1.9in}\times\,\iint \sum_{w \in H_M}g\big(x
h(v,w)\big)\nu(dv)\nu_\alpha(dx)
\end{align*}
which needs to be shown to converge to $0$ as $n \rightarrow \infty$. This is easy because $g \leq CI_{[-\infty,-\delta]\cup[\delta, \infty]}$ for some $C,\delta > 0$ (since $g \geq 0$ has compact support on $[-\infty,\infty] \setminus \{0\}$) which implies
\begin{align}
\iint \sum_{w \in H_M}g\big(x h(v,w)\big)\nu(dv)\nu_\alpha(dx) &\leq C \iint \sum_{w \in H_M} I_{\left(|x| \geq \delta/|h(v,w)|\right)} \nu_\alpha(dx)\nu(dv) \label{finiteness_of_the_integral} \\
&=C\delta^{-\alpha} \int_W \sum_{w \in H_M}|h(v,w)|^\alpha \nu(dv) < \infty,    \nonumber
\end{align}
and Lemma \ref{lemma_on_C_and_V} together with
\eqref{size_of_H_n} implies
\begin{align*}
&\left|\frac{1}{(cn)^p}\sum_{u \in H_{n+M}}\left(\frac{m(u^{-1},n+M)-m(u^{-1},n)}{n^q}\right)\right|\\
&=\frac{1}{(cn)^p}\sum_{u \in H_{n+M}}\left(\left(\frac{n+M}{n}\right)^q\frac{m(u,n+M)}{(n+M)^q}-\frac{m(u,n)}{n^q}\right)\\
&=\,o(1)+\frac{1}{c^p}\Bigg[\left(\frac{n+M}{n}\right)^{p+q}\frac{1}{(n+M)^p}\sum_{u \in H_{n+M}}\frac{m(u,n+M)}{(n+M)^q}\\
&\hspace{2.35in}-\frac{1}{n^p}\sum_{u \in
H_{n}}\frac{m(u,n)}{n^q}\Bigg] \rightarrow 0\,.
\end{align*}
This proves claim \eqref{difference_of_I_n's} which yields $\mathcal{I}_n
\leq \mathcal{I}^\prime_n + o(1)$. Similarly we can also get a lower
bound of $\mathcal{I}_n$ and establish that $\mathcal{I}_n \geq
\mathcal{I}^\prime_n + o(1)$. Hence, in order to complete the proof
of weak convergence of $\tilde{N}^{(2)}_n$ to $\tilde{N}_\ast$ under
the compact support assumption \eqref{assumption_cpt_support_h}, it
is enough to show that
\begin{align*}
&\mathcal{I}^\prime_n = \frac{1}{(cn)^p}\iint\sum_{u \in H_{n+M}}
\left(1-\exp\left(-\frac{m(u,n)}{n^q}\sum_{w \in A^\prime_n}g\big(x
h(v,w)\big)\right)\right)\\
&\hspace{3.8in} \nu(dv)\nu_\alpha(dx)
\end{align*}
converges to
\begin{equation}
l\frac{1}{c^p}\int_{C} \int_{|x|>0} \int_W
\left(1-\exp\left(-\mathcal{V}(y)\sum_{w \in H_M}
g(xh(v,w))\right)\right)\nu(dv)\nu_\alpha(dx)dy\,.
\label{limit_of_mathcal_I^prime_n}
\end{equation}
To this end we decompose the integral $\mathcal{I}^\prime_n$ into
two parts as follows:
\begin{align*}
&\;\;\;\;\mathcal{I}^\prime_n\\
&=\frac{1}{(cn)^p}\iint\sum_{u \in H_{n-M}} \left(1-\exp\left(-\frac{m(u,n)}{n^q}\sum_{w \in H_M}g\big(x h(v,w)\big)\right)\right)\\
&\hspace{3.9in}\nu(dv)\nu_\alpha(dx)\\
&\;\;\;\;\;+\frac{1}{(cn)^p}\iint\sum_{u \in B^\prime_n} \left(1-\exp\left(-\frac{m(u,n)}{n^q}\sum_{w \in A^\prime_n}g\big(x h(v,w)\big)\right)\right)\\
&\hspace{3.9in}\nu(dv)\nu_\alpha(dx)\\
&=: J^\prime_n+L^\prime_n
\end{align*}
for all $n > M$. Here $B^\prime_n=H_{n+M}\cap H_{n-M}^c$. For $1
\leq k \leq l$ let
\begin{align*}
&\;\;\;\;J^\prime_{k,n} \\
&= \frac{1}{(cn)^p}\iint\sum_{u \in F_{k,n-M}}
\left(1-\exp\left(-\frac{m(u,n)}{n^q}\sum_{w \in H_M}g\big(x
h(v,w)\big)\right)\right)\\
&\hspace{3.9in}\nu(dv)\nu_\alpha(dx)\,.
\end{align*}
Clearly, by \eqref{form_of_H_n_in_terms_of_F_k,n}, $J^\prime_n=\sum_{k=1}^lJ^\prime_{k,n}$. We will show that
each $J^\prime_{k,n}$, $1 \leq k \leq l$ converges to
\eqref{limit_of_mathcal_I^prime_n} except for the factor $l$.

Fix $k \in \{1,2,\ldots,l\}$. Repeating the argument in the proof of
\eqref{difference_of_I_n's} we obtain for all $n > M$,
\begin{align*}
&\;\;\;\;J^\prime_{k,n} \\
&=o(1)+\iint\frac{1}{(cn)^p}\sum_{u \in F_{k,n-M}} \left(1-e^{-\frac{m(u,n-M+L)}{n^q}\sum_{w \in H_M}g\big(x h(v,w)\big)}\right)\\
&\hspace{3.9in}\nu(dv)\nu_\alpha(dx)\\
&=o(1)\,+\,\left(\frac{n-M+L}{cn}\right)^p \times \\
&\;\;\;\;\;\;\iint\frac{1}{(n-M+L)^p}\sum_{\alpha \in Q^{(k)}_{n-M}}\left(1-e^{-\frac{m(x_k+U\alpha,n-M+L)}{n^q}\sum_{w \in H_M}g\big(x h(v,w)\big)}\right)\\
&\hspace{3.9in}\nu(dv)\nu_\alpha(dx)\\
\intertext{which can be estimated using \eqref{inclusions_of_Q_n} as
follows:}
&\leq o(1)\,+\,\left(\frac{n-M+L}{cn}\right)^p \times \\
&\;\;\;\int_{|x|>0}\int_W\int_{C}\left(1-e^{-\frac{m(x_k+\sum_{i=1}^{p}[(n-M+L)y_i]\bar{u}_i,n-M+L)}{n^q}\sum_{w \in H_M}g\big(x h(v,w)\big)}\right)\\
&\hspace{3.7in}dy\,\nu(dv)\,\nu_\alpha(dx)\,.\\
\intertext{By Lemma \ref{lemma_on_C_and_V} there is a constant
$\kappa >0$ such that the above integrand sequence is dominated by}
&\;\;\;\;\;\;\;\;\;\;\;\;\;1-\exp\left\{-\kappa\sum_{w \in H_M}g\big(x h(v,w)\big)\right\}\\
\intertext{which can be shown to be integrable using the inequality $1-e^{-x} \leq x,\;x>0$ and the arguments given in \eqref{finiteness_of_the_integral}.
Hence Lemma
\ref{lemma_on_C_and_V} together with the dominated convergence theorem
yields}
&\;\;\;\int_{|x|>0}\int_W\int_{C}\left(1-e^{-\frac{m(x_k+\sum_{i=1}^{p}[(n-M+L)y_i]\bar{u}_i,n-M+L)}{n^q}\sum_{w \in H_M}g\big(x h(v,w)\big)}\right)\\
&\hspace{3.7in}dy\,\nu(dv)\,\nu_\alpha(dx)\\
&\rightarrow \int_{C} \int_{|x|>0} \int_W
\left(1-\exp\left(-\mathcal{V}(y)\sum_{w \in H_M}
g(xh(v,w))\right)\right) \nu(dv)
\nu_\alpha(dx)dy\,.\\
\intertext{This shows}
&\;\;\limsup_{n\rightarrow \infty} J^\prime_{k,n} \\
& \leq \frac{1}{c^p}\int_{C} \int_{|x|>0} \int_W
\left(1-\exp\left(-\mathcal{V}(y)\sum_{w \in H_M}
g(xh(v,w))\right)\right) \nu(dv)
\nu_\alpha(dx)dy\,.\\
\intertext{Similarly we can also prove that}
&\;\;\liminf_{n\rightarrow \infty} J^\prime_{k,n} \\
& \geq \frac{1}{c^p}\int_{C} \int_{|x|>0} \int_W
\left(1-\exp\left(-\mathcal{V}(y)\sum_{w \in H_M}
g(xh(v,w))\right)\right) \nu(dv)
\nu_\alpha(dx)dy\,.\\
\intertext{Hence, $J^\prime_n$ converges to
\eqref{limit_of_mathcal_I^prime_n} as $n\rightarrow \infty$. To
establish the weak convergence of $\tilde{N}^{(2)}_n$ when $h$ is
compactly supported it remains to prove that $L^\prime_n \rightarrow
0$ as $n\rightarrow \infty$. This is easy because}
&\;\;L^\prime_n \\
&\leq \frac{1}{(cn)^p}\iint\sum_{u \in B^\prime_n} \left(1-\exp\left(-\frac{m(u,n)}{n^q}\sum_{w \in H_M}g\big(x h(v,w)\big)\right)\right)\nu(dv)\nu_\alpha(dx)\\
&=\frac{1}{(cn)^p}\iint\sum_{u \in H_{n+M}} \left(1-\exp\left(-\frac{m(u,n)}{n^q}\sum_{w \in H_M}g\big(x h(v,w)\big)\right)\right)\\
&\hspace{3.9in}\nu(dv)\nu_\alpha(dx)\\
&-\frac{1}{(cn)^p}\iint\sum_{u \in H_{n-M}} \left(1-\exp\left(-\frac{m(u,n)}{n^q}\sum_{w \in H_M}g\big(x h(v,w)\big)\right)\right)\\
&\hspace{3.9in}\nu(dv)\nu_\alpha(dx)\\
& \rightarrow 0
\end{align*}
since the first term can also be shown to converge to the same limit
as the second term by the exact same argument as above.

To remove the assumption of compact support on the function $h$, for
a general $h \in L^\alpha(\nu \otimes \tau)$ define
\begin{equation}
h_M(v,u)=h(v,u)I_{H_M}(u),\;\;\;M \geq 1. \label{defn_of_h_M}
\end{equation}
Notice that each $h_M$ satisfies \eqref{assumption_cpt_support_h}
and that $h_M \rightarrow h$ almost surely as well as in
$L^\alpha(\nu \otimes \tau)$ as $M \rightarrow \infty$. Denote
\begin{equation}
\tilde{N}^{(2,M)}_n=\sum_{i=1}^\infty \sum_{t \in
H_n}\frac{m(t,n)}{n^q}\,\delta_{(cn)^{-p/\alpha}j_ih_M(v_i,u_i\oplus
t)}\,, \label{defn_of_tilde_N^2,M}
\end{equation}
for $M,n \geq 1$, and
\begin{equation}
\tilde{N}^{(M)}_\ast = \sum_{i=1}^\infty \sum_{u \in H}
\mathcal{V}(\xi_i) \delta_{j_ih_M(v_i,u)}\,,\;\;\;M \geq 1
\label{defn_of_tilde_N^M_star}
\end{equation}
with the notations as above. We already know that for every $M \geq
1$, $\tilde{N}^{(2,M)}_n \Rightarrow \tilde{N}^{(M)}_\ast$ weakly in
the space $\mathcal{M}$ as $n \rightarrow \infty$. Therefore, to
establish $\tilde{N}^{(2)}_n \Rightarrow \tilde{N}_\ast$, it is
enough to show two things:
\begin{equation}
\tilde{N}^{(M)}_\ast \Rightarrow \tilde{N}_\ast\;\;\mbox{ weakly as
}M \rightarrow \infty  \label{limit_of_tilde_N^M_star}
\end{equation}
and
\begin{equation}
\lim_{M\rightarrow \infty} \limsup_{n\rightarrow \infty}
P(|\tilde{N}_n^{(2,M)}(g)-\tilde{N}^{(2)}_n(g)|>\epsilon)=0
\label{difference_of_tilde_N^2,K_and_tilde_N^2}
\end{equation}
for all $\epsilon > 0$ and for every non-negative continuous
function $g$ with compact support on $[-\infty,\infty] \setminus \{0\}$.

Claim \eqref{limit_of_tilde_N^M_star} is easy since the Laplace functional
of $\tilde{N}^{(M)}_\ast$, which is obtained by replacing $h$ in
\eqref{form_of_laplace_functional_of_tilde_N_star} by $h_M$,
converges by the dominated convergence theorem to
\eqref{form_of_laplace_functional_of_tilde_N_star} for every
non-negative continuous function $g$ with compact support on
$[-\infty,\infty] \setminus \{0\}$. The proof of
\eqref{difference_of_tilde_N^2,K_and_tilde_N^2} is along the same
lines as the proof of the corresponding limit (namely $(3.13)$) in \cite{resnick:samorodnitsky:2004}. Using
similar calculations we have
\begin{align*}
&\;\;E|\tilde{N}_n^{(2,M)}(g)-\tilde{N}^{(2)}_n(g)|\\
&=\sum_{t \in H_n} \frac{m(t,n)}{n^q}E\left(\sum_{i=1}^{\infty}g((cn)^{-p/\alpha}j_ih(v_i,u_i\oplus t))I(N(u_i\oplus t)>M)\right)\\
&=\left(\frac{1}{(cn)^p}\sum_{t\in
H_n}\frac{m(t,n)}{n^q}\right)\int_W\int_{|x|>0}\sum_{u \in
H_M^c}g(xh(v,u))\nu_\alpha(dx)\nu(dv)\,.
\intertext{Repeating the argument in \eqref{finiteness_of_the_integral}, the integral}
&\int_W\int_{|x|>0}\sum_{u \in H_M^c}g(xh(v,u))\nu_\alpha(dx)\nu(dv)
\intertext{can be shown to be bounded by}
&C \delta^{-\alpha}\int_W \sum_{u \in H_M^c} |h(u,v)|^\alpha \nu(dv)
\end{align*}
which converges to $0$ as $M \rightarrow \infty$. Hence, by Lemma
\ref{lemma_on_C_and_V},
\eqref{difference_of_tilde_N^2,K_and_tilde_N^2} follows and so does
$\tilde{N}^{(2)}_n \Rightarrow \tilde{N}_\ast$ without the
assumption of compact support.

To complete the proof of the theorem, we need to prove (with $\rho$
being the vague metric on $\mathcal{M}$) that for all $\epsilon >0$
\[
P[\rho(\tilde{N}_n,\tilde{N}^{(2)}_n)>\epsilon] \rightarrow 0
\;\;\;(n \rightarrow \infty)
\]
and for this, it suffices to show that for every non-negative
continuous function $g$ with compact support on
$[-\infty,\infty] \setminus \{0\}$,
\begin{align}
&\;\;\;\;\;\;\;\;\;P(|\tilde{N}_n(g)-\tilde{N}^{(2)}_n(g)|>\epsilon)     \nonumber \\
&\;\;\;\;\;=P\left(\left|\sum_{t \in H_n} \frac{m(t,n)}{n^q}\left(g\left(\frac{X_t}{(cn)^{p/\alpha}}\right)-\sum_{i=1}^\infty g\left(\frac{j_i h(v_i,u_i\oplus t)}{(cn)^{p/\alpha}}\right)\right) \right| >\epsilon \right)\label{difference_of_tilde_N_and_tilde_N^2}\\
&\;\;\;\;\;\rightarrow 0 \nonumber
\end{align}
as $n \rightarrow \infty$. By Lemma \ref{lemma_on_C_and_V},
\eqref{difference_of_tilde_N_and_tilde_N^2} would follow from
\begin{equation}
\;\;\;\;\;\;\;\;P\left(\left|\sum_{t \in
H_n}\left(g\left(\frac{X_t}{(cn)^{p/\alpha}}\right)-\sum_{i=1}^\infty
g\left(\frac{j_i h(v_i,u_i\oplus t)}{(cn)^{p/\alpha}}\right)\right)
\right| >\epsilon/\kappa_0 \right)\rightarrow
0\,.\label{difference_of_tilde_N_and_tilde_N^2_restated}
\end{equation}
Here $\kappa_0$ is as in \eqref{defn_of_kappa_0}. Once again,
following verbatim the proof of $(3.14)$ in
\cite{resnick:samorodnitsky:2004}, we can establish
\eqref{difference_of_tilde_N_and_tilde_N^2_restated} and complete
the proof of Theorem \ref{thm_point_process_Gdiss}.

\section{An Example} \label{sec_example}

We end this paper by considering a simple example and computing
the weak limit of the corresponding random measure (properly
normalized $\{N_n\}$) using Theorem \ref{thm_point_process_Gdiss}.
This will help us understand the result as well as get used to the
notations.

\begin{example}
\textnormal{Suppose $d=2$, and define the $\mathbb{Z}^2$-action
$\{\phi_{(t_1,t_2)}\}$ on $S=\mathbb{R}$ as
\[
\phi_{(t_1,t_2)}(x)=x+t_1-t_2\,.
\]
Take any $f \in L^\alpha(S,\mu)$ where $\mu$ is the Lebesgue measure
on $\mathbb{R}$ and define a stationary $S\alpha S$ random field
$\{X_{(t_1,t_2)}\}$ as follows
\[
X_{(t_1,t_2)}=\int_{\mathbb{R}} f\big(\phi_{(t_1,t_2)}(x)\big)\,
M(dx),\;\;\;t_1,t_2 \in \mathbb{Z}\,,
\]
where $M$ is an $S \alpha S$ random measure on $\mathbb{R}$ with
control measure $\mu$. Note that the above representation of
$\{X_{(t_1,t_2)}\}$ is of the form $(\ref{repn_integral_stationary})$ generated by a
measure preserving conservative action with $c_{(t_1,t_2)} \equiv
1$.}

\textnormal{In this case, using the notations as above, we have
\begin{align*}
K&=\{(t_1,t_2)\in \mathbb{Z}^2:\,t_1=t_2\} \\
\intertext{which implies $A \simeq \mathbb{Z}^2/K \simeq
\mathbb{Z}\,,$ and} F&=\{(t_1,0):\,t_1 \in \mathbb{Z}\}\,.
\end{align*}
In particular we have $p=q=l=1$, and
\[
U=\left[\begin{array}{c}
1\\
0
\end{array}
\right],\;V=\left[\begin{array}{c}
1\\
1
\end{array}\right]
\]
so that
\begin{align*}
C&=\{y \in \mathbb{R}:\, \mbox{ there exists }\lambda \in \mathbb{R} \mbox{ such that }\|Uy+V\lambda\|_\infty \leq 1\}\\
&=\{y \in \mathbb{R}:\,|y+\lambda| \leq 1 \mbox{ for some }\lambda \in [-1,1]\}\,=\,[-2,2]\,.\\
\intertext{For all $y \in C=[-2,2]$ we have} P_y&=\{\lambda \in
[-1,1]:\,|y+\lambda| \leq 1\}=\left\{
\begin{array}{ll}
[-(1+y),1] &\;\;\;y \in [-2,0) \\
\,\,[-1,1-y] &\;\;\;y \in [0,2]
\end{array}
\right.\\
\intertext{which yields} &\mathcal{V}(y)=2-|y|\,,\;\;\;y \in
[-2,2]\,.
\end{align*}}
\textnormal{\;\;\;\,Clearly, $\{X_{(t_1,0)}\}_{t_1 \in \mathbb{Z}}$
is a stationary $S\alpha S$ process generated by a dissipative $\mathbb{Z}$-action
$\{\phi_{(t_1,0)}\}_{t_1\in\mathbb{Z}}$. Hence, by Theorem $4.4$ in
\cite{rosinski:1995}, there is a $\sigma$-finite standard measure
space $(W,\nu)$ and a function $h \in L^\alpha(W \times \mathbb{Z},
\nu \otimes \zeta_{\mathbb{Z}})$ such that
\[
X_{(t_1,0)} \eqdef \int_{W \times \mathbb{Z}}
h(v,t_1+s)M(dv,ds)\,,\;\;\;t_1\in \mathbb{Z}\,.
\]
Here $\zeta_{\mathbb{Z}}$ is the counting measure on $\mathbb{Z}$, and
$M$ is an $S \alpha S$ random measure on $W \times \mathbb{Z}$ with
control measure $\nu \otimes \zeta_{\mathbb{Z}}$. Let
$$\sum_{i=1}^{\infty}\delta_{(j_i,v_i,\xi_i)} \sim
PRM\big(\nu_\alpha \otimes \nu \otimes
\frac{1}{4}Leb|_{[-2,2]}\big)$$ be a Poisson random measure on
$([-\infty,\infty] \setminus \{0\}) \times W \times [-2,2]$. In this example,
$c=\big(l|C|\big)^{1/p}=4$ and
$$\tilde{N}_n = n^{-1}
\sum_{|t_1|,\,|t_2| \leq n}
\delta_{(4n)^{-1/\alpha}X_{(t_1,t_2)}}\,,\;\;\;n=1,2,\ldots\,.$$
Since $\{\phi_u\}_{u \in F}$ is a dissipative group action and
\eqref{assumption_on_c_t_for_t_in_K} holds in this case, we can use
Theorem \ref{thm_point_process_Gdiss} and conclude that
\[
\tilde{N}_n \Rightarrow \sum_{i=1}^{\infty}\sum_{t_1 \in
\mathbb{Z}}\big(2-|\xi_i|\big)\,\delta_{j_ih(v_i,t_1)}
\]
weakly in the space $\mathcal{M}$.}
\end{example}

\begin{remark}
\textnormal{Note that $\tilde{N}_n$ can also be written as follows:
\[
\tilde{N}_n=\sum_{k=-2n}^{2n}\left(2-\frac{|k|}{n}+\frac{1}{n}\right)\,\delta_{(4n)^{-1/\alpha}Y_k}
\]
where $Y_k=X_{(k,0)}$. Only a few (a Poisson number) of the $Y_k$'s are not driven to
zero by the normalization $b_n=(4n)^{-1/\alpha}$. By stationarity,
each of these rare $k$'s should be distributed uniformly in
$\{-2n,-2n+1,\ldots,2n\}$ which along with Theorem \ref{thm_point_process_Z^d_diss} provides an intuitive justification of the above weak
limit of $\tilde{N}_n$.}
\end{remark}

\section*{Acknowledgements}

The author is immensely grateful to Gennady Samorodnitsky for the extremely useful discussions and valuable suggestions which contributed a lot to this work, to Paul Embrechts for the support and the suggestions during the author's stay at RiskLab, and to Jan Rosi\'nski for Remark \ref{remark_jan} which significantly improved this paper. He is also thankful to Johanna Neslehova, Augusto Teixeira and the anonymous referees for their comments.

\end{document}